\documentclass[review]{elsarticle}

\usepackage[fleqn]{mathtools}
\usepackage{amsmath, amssymb}
\usepackage{tikz}
\usepackage{tcolorbox}
\usepackage[linesnumbered,lined,boxed,commentsnumbered]{algorithm2e}
\usepackage{subfigure}
\usepackage{multirow}
\newtheorem{mylemma}{Lemma}
\newtheorem{proof}{Proof}

\biboptions{authoryear}

\makeatletter
\def\ps@pprintTitle{%
   \let\@oddhead\@empty
   \let\@evenhead\@empty
   \def\@oddfoot{\reset@font\hfil\thepage\hfil}
   \let\@evenfoot\@oddfoot
}
\makeatother

\begin{document}

\begin{frontmatter}

\title{Computing non-stationary ($s, S$) policies using mixed integer linear programming}

\author[1]{Mengyuan Xiang}
\ead{mengyuan.xiang@ed.ac.uk}

\author[1]{Roberto Rossi\corref{cor1}}
\ead{roberto.rossi@ed.ac.uk}

\author[1]{Belen Martin-Barragan}
\ead{belen.martin-barragan@ed.ac.uk}

\author[2]{S. Armagan Tarim}
\ead{at@cankaya.edu.tr}

\cortext[cor1]{Corresponding author}

\address[1]{Business School, University of Edinburgh, Edinburgh, United Kingdom}

\address[2]{Department of Management, Cankaya University, Ankara, Turkey}

\begin{abstract}
This paper addresses the single-item single-stocking location stochastic lot sizing problem under the $(s, S) $ policy. We first present a mixed integer non-linear programming (MINLP) formulation for determining near-optimal $(s, S)$ policy parameters. To tackle larger instances, we then combine the previously introduced MINLP model and a binary search approach. These models can be reformulated as mixed integer linear programming (MILP) models which can be easily implemented and solved by using off-the-shelf optimisation software. Computational experiments demonstrate that optimality gaps of these models are around $0.3\%$ of the optimal policy cost and computational times are reasonable. 
\end{abstract}


\end{frontmatter}


\section{Introduction}
\noindent
Stochastic lot sizing is an important research area in inventory theory. One of the landmark studies is \cite{Scarf1960} which proved the optimality of $(s, S)$ policies for a class of dynamic inventory models. The $(s, S)$ policy features two control parameters: $s$ and $S$. Under this policy, the decision maker checks the opening inventory level at the beginning of each time period: if it drops to or below the reorder point $s$, then a replenishment should be placed to reach the order-up-to-level $S$. Unfortunately, computing optimal $(s, S)$ policy parameters remains a computationally intensive task.

In the literature, studies on $(s, S)$ policy can be categorized into stationary and non-stationary. A number of attempts have been made to compute stationary $(s, S)$ policy parameters, e.g. \citep{i1963, vw1965, ac1978, s1977, s1982, fz1984, zf1991, fx2000}. However, in reality, there has been an increasing recognition that lot-sizing studies need to be undertaken for non-stationary environments \citep{citeulike:13509695}. Additionally, only two studies investigated computations of $(s, S)$ policy under non-stationary stochastic demand \citep{a1981, bm1999}. This motivates our work on non-stationary $(s, S)$ policy. 

\cite{a1981} adopted the ``least cost per unit time'' approach in selecting order-up-to-levels and reorder points under a penalty cost scheme. Decision makers first determine desired cycle lengths and order-up-to-levels. Then, reorder points are decided by means of a trade-off analysis between expected costs per period in cases of ordering and not ordering. 

As \cite{bm1999} pointed out, \cite{a1981} is computationally expensive because of the convolutions of demand distributions. In contrast, \cite{bm1999} proposed a stationary approximation heuristic for computing optimal $(s, S)$ policy parameters. Firstly, decision makers precompute pairs of $(s, S)$ values for various demand parameters and tabulate results. Here, a large number of efficient algorithms exist for generating the stationary table, e.g. \citep{fz1984, zf1991, fx2000}. Secondly, order-up-to-levels and reorder points can be read from stationary tables by averaging the demand parameters over an estimate of the expected time between two orders. However, this algorithm relies upon complex code, particularly for generating stationary tables. 

Unfortunately, both these works \citep{a1981, bm1999} do not provide a satisfactory solution to the problem: they rely on ad-hoc computer coding and provide relatively large optimality gaps. A recent computational study \cite{dktr2016} estimated the optimality gap of \citep{a1981, bm1999} at $3.9\%$ and $4.9\%$, respectively. These drawbacks motivate our work in finding a heuristic method for computing $(s, S)$ policy parameters which does not need computer coding and can provide better optimality gaps. 

In this paper, we therefore introduce a new modelling framework to compute near-optimal $(s, S)$ policy parameters. In particular, we consider a single-item single-stocking location stochastic lot-sizing problem under non-stationary demand, fixed and unit ordering cost, holding cost and penalty cost. In contrast to other approaches in the literature, our models can be easily implemented and solved by using off-the-shelf software such as IBM ILOG optimisation studio. We make the following contributions to literature on stochastic lot-sizing.
\begin{itemize}
\item We introduce the first mixed integer non-linear programming (MINLP) model to compute near-optimal $(s, S)$ policy parameters.
\item  We show that this model can be reformulated as a mixed integer linear programming (MILP) model by piecewise linearising the cost function; this reformulation can be solved by using off-the-shelf software.
\item To tackle larger instances, we combine the previously introduced MINLP model and a binary search procedure.
\item Computational experiments demonstrate that optimality gaps of our models are tighter than existing algorithms \citep{a1981, bm1999} in the literature, and computational times of our models are reasonable.
\end{itemize}

The rest of this paper is organised as follows. Section \ref{problemdescription} describes  problem settings and a stochastic dynamic programming (SDP) formulation. Section \ref{approximation} discusses the notion of $K$-convexity and introduces relevant $K$-convex cost functions which are approximated by an MINLP model in Section \ref{minlpmodel}. Section \ref{milpmodel} presents an MINLP heuristic for approximating $(s, S)$ policy parameters. Section \ref{binarysearchmodel} introduces an alternative binary search approach for computing $(s, S)$ policy parameters. A detailed computational study is given in Section \ref{computationalstudy}. Finally, we draw conclusions in Section \ref{conclusion}.

\section{Problem description}\label{problemdescription}
\noindent
We consider a single-item single-stocking location inventory management system over a $T$-period planning horizon.  We assume that orders are placed at the beginning of each time period, and delivered instantaneously. There exist ordering costs $c(\cdot)$ comprising a fixed ordering cost $K$ for placing an order, and a linear ordering cost $c$ proportional to order quantity $Q$. Demands $d_t$ in each period $t=1, \ldots, T$ are independent random variables with known probability distributions. At the end of period $t$, a linear holding cost $h$ is charged on every unit carried from one period to the next; and a linear penalty cost $b$ is occurred for each unmet demand at the end of each time period.

For a given period $t=\{1, \ldots, T\}$, let $I_{t-1}$ denote the opening inventory level and $Q_t$ represent the order quantity. Then the immediate cost of period $t$ can be expressed as
\begin{equation}\label{e1}
f_t(I_{t-1},Q_t)=c(Q_t)+\text{E}[h \max(I_{t-1}+Q_t-d_t,0)+b \max(d_t-I_{t-1}-Q_t,0)], 
\end{equation}
where E denotes the expectation taken with respect to the random demand $d_t$. Additionally, the ordering cost  $c(Q_t)$ is defined as:
\[c(Q_t)=\begin{cases}
K+c~Q_t, &\texttt{ $Q_t>0$} \\
0, & \texttt{ $Q_t = 0$}
\end{cases}\]

Let $C_t(I_{t-1})$ represent the expected total cost of an optimal policy over periods $t, \ldots, T$ when the initial inventory level at the beginning of period $t$ is $I_{t-1}$. We model the problem as a stochastic dynamic program \citep{Bellman:1957} via the following functional equation
\begin{equation}\label{minsdp}
C_t(I_{t-1})=\min_{Q_t} \left\{ f_t(I_{t-1},Q_t)+\text{E}[C_{t+1}(I_{t-1}+Q_t-d_t)]\right\}
\end{equation}
where 
\[C_T(I_{T-1})=\min_{Q_t}  f_T(I_{T-1},Q_T)\]
represents the boundary condition.

\section{The optimality of $(s, S)$ policies in stochastic lot sizing}\label{approximation}
\noindent
\cite{Scarf1960} proved that the optimal policy in the dynamic inventory problem is always of the $(s, S)$ type based on a study of the function
\begin{equation}
\label{sgfunction}
G_t(y)=cy+\text{E}[h \max(y-d_t,0)+b \max(d_t-y,0)]+\text{E}[C_{t+1}(y-d_t)],
\end{equation}
where $y$ is the stock level immediately after purchases are delivered.

Since we consider a non-stationary environment, values of the $(s, S)$ policy parameters will depend on the given period $t$. Let $(s_t, S_t)$ denote the policy parameters for period $t$. Function $G_t(y)$ can be used to define the policy parameters $(s_t, S_t)$ and prove their optimality. In particular, the order-up-to-level $S_t$ is defined as the value minimising $G_t(y)$; whereas the parameters $s_t$ is given by the value $s_t<S_t$ such that $K+G_t(S_t)=G_t(s_t)$. $K$-convexity of the function $G_t(y)$ ensures the uniqueness of $s_t$ and $S_t$ \citep{Scarf1960}.

{\bf Example. }
We illustrate the concepts introduced on a 4-period example. Demand $d_t$ is normally distributed in each period $t$ with mean $\mu_t \in\{20, 40, 60, 40\}$, for $t= 1,\ldots, 4$ respectively. Standard deviation $\sigma_t$ of demand in period $t$ is equal to $0.25\mu_t$. Other parameters are $K=100$, $h=1$, $b=10$, and $c=0$. We plot $G_1(y)$ in Fig. \ref{Gfunction} for initial inventory levels $y \in (0, 200)$. The expected total costs $G_1(y)$ are obtained via SDP. The order-up-to-level is $S_1=70$ and the minimised expected total cost $G_1(S_1)=262.5839$; the reorder point is $s_1=14$ and the corresponding cost $G_1(s_1)=362.5839$. Note that $G_1(s_1)=G_1(S_1)+K$. The optimal policy is to order to $70$ if the initial inventory $y<14$; otherwise not to order.

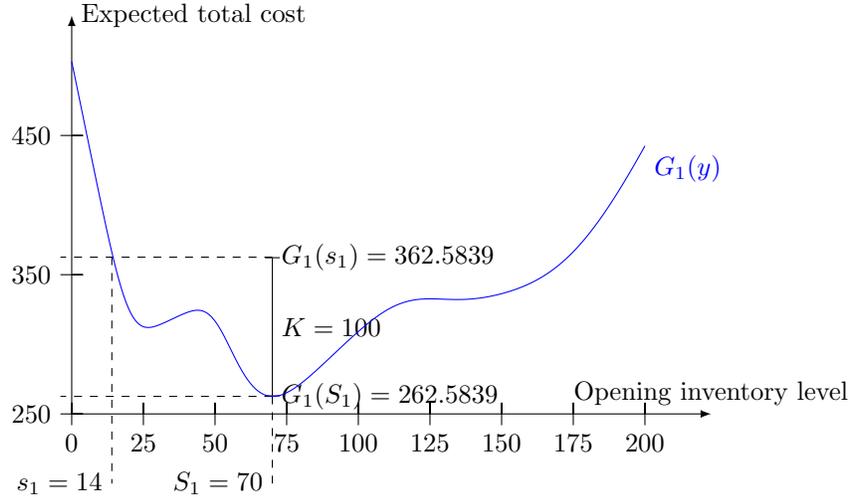
\begin{figure}[!htp]
\centering
\begin{tikzpicture}[x=0.0380952380952381cm, y=0.018518518518518517cm]
\draw [-latex] ([xshift=-0mm] 0.0,0) -- ([xshift=5mm] 210.0,0) node[above] {Opening inventory level};
\draw (0.0,0) -- +(0mm,1.5mm) -- +(0mm,-1.5mm) node[below] {0};
\draw (25.0,0) -- +(0mm,1.5mm) -- +(0mm,-1.5mm) node[below] {25};
\draw (50.0,0) -- +(0mm,1.5mm) -- +(0mm,-1.5mm) node[below] {50};
\draw (75.0,0) -- +(0mm,1.5mm) -- +(0mm,-1.5mm) node[below] {75};
\draw (100.0,0) -- +(0mm,1.5mm) -- +(0mm,-1.5mm) node[below] {100};
\draw (125.0,0) -- +(0mm,1.5mm) -- +(0mm,-1.5mm) node[below] {125};
\draw (150.0,0) -- +(0mm,1.5mm) -- +(0mm,-1.5mm) node[below] {150};
\draw (175.0,0) -- +(0mm,1.5mm) -- +(0mm,-1.5mm) node[below] {175};
\draw (200.0,0) -- +(0mm,1.5mm) -- +(0mm,-1.5mm) node[below] {200};
\draw [-latex] ([yshift=-0mm] 0,0.0) -- ([yshift=3mm] 0, 270.0) node[right] {Expected total cost};
\draw (0,0.0) -- +(1.5mm,0mm) -- +(-1.5mm,0mm) node[left] {250};
\draw (0,100.0) -- +(1.5mm,0mm) -- +(-1.5mm,0mm) node[left] {350};
\draw (0,200.0) -- +(1.5mm,0mm) -- +(-1.5mm,0mm) node[left] {450};

\draw [color=black, mark= , style=dashed] (70,12.5839) -- (70,-50) node[left] {$S_1=70$};
\draw [color=black, mark= , style=dashed] (14,112.5839) -- (14,-50) node[left] {$s_1=14$};

\draw [color=black, mark= , style=solid, |-|] (70,112.5839) node[right] {$G_1(s_1)=362.5839$} -- (70,12.5839) node[right] {$G_1(S_1)=262.5839$} node[pos=0.5,right]{$K=100$};
\draw [color=black, mark= , style=dashed] (70,12.5839) -- (-1.5mm,12.5839);
\draw [color=black, mark= , style=dashed] (70,112.5839) -- (-1.5mm,112.5839);
\draw [smooth, color=blue, mark= , style=solid] plot coordinates {
(0, 253.0985133)
(1, 243.0990424)
(2, 233.1002282)
(3,	223.1027871)
(4,	213.1081048)
(5,	203.1187484)
(6,	193.1392723)
(7,	183.1774091)
(8,	173.2457154)
(9,	163.3636806)
(10, 153.5601892)
(11	, 143.8760714)
(12	, 134.3662915)
(13, 125.1011707)
(14, 116.165976)
(15, 107.6583027)
(16, 99.68296406)
(17, 92.34456423)
(18, 85.73847716)
(19, 79.94145151)
(20, 75.0033453)
(21, 70.94145151)
(22, 67.73847716)
(23, 65.34456423)
(24, 63.68296406)
(25, 62.65830268)
(26, 62.16597601)
(27, 62.10117068)
(28, 62.36629146)
(29, 62.87607137)
(30, 63.5600046)
(31, 64.36287635)
(32, 65.24357091)
(33, 66.17246024)
(34, 67.12865938)
(35, 68.09710261)
(36, 69.06574092)
(37, 70.02291631)
(38, 70.95488918)
(39, 71.84349815)
(40, 72.66401071)
(41, 73.38335025)
(42, 73.95901997)
(43, 74.33913432)
(44, 74.46395957)
(45, 74.26922138)
(46, 73.69115335)
(47,	 72.6728824)
(48, 71.17136191)
(49, 69.16378776)
(50, 66.65236669)
(51, 63.66650404)
(52, 60.26191182)
(53, 56.51670934)
(54, 52.52514768)
(55, 48.3899929)
(56, 44.21475584)
(57, 40.09684214)
(58, 36.12237561)
(59, 32.3630303)
(60, 28.87480724)
(61, 25.69840004)
(62, 22.86064736)
(63, 20.37656259)
(64, 18.25151972)
(65, 16.48330934)
(66, 15.06391322)
(67, 13.98095271)
(68, 13.21883454)
(69, 12.75964926)
(70, 12.58388394)
(71, 12.67100198)
(72, 12.99992808)
(73, 13.54946284)
(74, 14.29863993)
(75, 15.22703176)
(76, 16.31500582)
(77, 17.54393194)
(78, 18.89634126)
(79, 20.35603795)
(80, 21.90816635)
(81, 23.53923672)
(82, 25.23711423)
(83, 26.99097589)
(84, 28.79124099)
(85, 30.62948017)
(86, 32.49830846)
(87,	 34.39126704)
(88, 36.30269816)
(89, 38.22761693)
(90, 40.16158328)
(91, 42.10057692)
(92, 44.04087734)
(93, 45.978951)
(94, 47.91134712)
(95, 49.83460367)
(96, 51.74516464)
(97,	 53.63931013)
(98, 55.51310022)
(99, 57.36233405)
(100, 59.18252515)
(101, 60.96889425)
(102, 62.71638021)
(103, 64.41966993)
(104, 66.0732472)
(105, 67.67146035)
(106, 69.20860754)
(107, 70.67903817)
(108, 72.07726798)
(109, 73.39810459)
(110, 74.63677967)
(111, 75.78908326)
(112, 76.85149528)
(113, 77.82130896)
(114, 78.69674105)
(115, 79.47702368)
(116, 80.1624734)
(117, 80.75453351)
(118, 81.25578704)
(119, 81.66993843)
(120, 82.00176378)
(121, 82.2570305)
(122, 82.4423887)
(123, 82.56523782)
(124, 82.63357314)
(125, 82.65581765)
(126, 82.64064531)
(127, 82.59680203)
(128, 82.53293079)
(129, 82.45740681)
(130, 82.37818816)
(131, 82.30268652)
(132, 82.23766152)
(133, 82.1891412)
(134, 82.16236982)
(135, 82.16178319)
(136, 82.19101055)
(137, 82.252901)
(138, 82.34957197)
(139, 82.48247615)
(140, 82.65248329)
(141, 82.85997283)
(142, 83.10493343)
(143, 83.3870654)
(144, 83.70588262)
(145, 84.06081049)
(146, 84.45127743)
(147, 84.87679748)
(148, 85.33704239)
(149, 85.83190199)
(150, 86.36153209)
(151, 86.9263898)
(152, 87.5272562)
(153, 88.16524715)
(154, 88.84181265)
(155, 89.55872614)
(156, 90.31806447)
(157, 91.12218013)
(158, 91.97366682)
(159, 92.87531981)
(160, 93.83009226)
(161, 94.84104896)
(162, 95.91131844)
(163, 97.04404477)
(164, 98.24233994)
(165, 99.50923771)
(166, 100.8476499)
(167, 102.2603255)
(168, 103.7498136)
(169, 105.31843)
(170, 106.9682288)
(171, 108.7009778)
(172, 110.5181393)
(173, 112.4208557)
(174, 114.4099392)
(175, 116.4858674)
(176, 118.648782)
(177, 120.8984927)
(178, 123.2344849)
(179, 125.6559305)
(180, 128.1617022)
(181, 130.750391)
(182, 133.4203254)
(183, 136.1695932)
(184, 138.996064)
(185, 141.8974141)
(186, 144.8711506)
(187, 147.9146371)
(188, 151.025119)
(189, 154.1997476)
(190, 157.4356051)
(191, 160.7297275)
(192, 164.0791265)
(193, 167.4808109)
(194, 170.9318052)
(195, 174.429168)
(196, 177.9700075)
(197, 181.5514963)
(198, 185.1708833)
(199, 188.8255054)
(200, 192.5127956)}
node[below right] {$G_1(y)$};
\end{tikzpicture}
\caption{Plot of $G_1(y)$}
\label{Gfunction}
\end{figure}

\section{MINLP approximation of Scarf's $G_t(y)$ function} \label{minlpmodel}
\noindent
In this section, we exploit an MINLP model to approximate the function $G_t(y)$ in Eq. (\ref{sgfunction}). Our model follows the control policy known as ``static-dynamic uncertainty'' strategy, originally introduced in \cite{bt1988}. Under this strategy, the timing of orders and order-up-to-levels are expected to be determined at the beginning of the planning horizon, while associated order quantities are decided upon only when orders are issued. As illustrated in \cite{rkt2015}, this strategy provides a cost performance which is close to the optimal ``dynamic uncertainty'' strategy. However, optimal $(s, S)$ parameters cannot be immediately derived from existing mathematical programming models operating under a static-dynamic uncertainty strategy, such as \cite{citeulike:12317242}, and \cite{rkt2015}. We next illustrate how a model operating under a static-dynamic uncertainty strategy can be used to approximate the function $G_t(y)$ in Eq. (\ref{sgfunction}). 

Consider a random variable $\omega$ and a scalar variable $x$. The first order loss function is defined as $L(x,\omega)=\text{E}[\max(\omega-x,0)]$, where E denotes the expected value with respect to the random variable $\omega$. The complementary first order loss function is defined as $\hat{L}(x,\omega)=\text{E}[\max(x-\omega,0)]$. Like \cite{rkt2015}, we will model non-linear holding and penalty costs by means of this function.

Consider three sets of decision variables: $\tilde{I}_t$, the expected closing inventory level at the end of period $t$, with $I_0$ denoting the initial inventory level; $\delta_t$, a binary variable which is set to one if an order is placed in period $t$; $P_{jt}$, a binary variable which is set to one if and only if the most recent replenishment before period $t$ was issued in period $j$. Let $\tilde{d}_{jt}$ denote the expected value of the demand over periods $j,\ldots,t$, i.e. $\tilde{d}_{jt}=\tilde{d}_j+\cdots+\tilde{d}_t$. Decision variables $H_{t} \geq 0$ and $B_t \geq 0$ for $t=1, \ldots, T$ represent end of period $t$ expected excess inventory and back-orders, respectively. An MINLP formulation for the non-stationary stochastic lot-sizing problem, obtained following the modeling strategy in  \cite{rkt2015}, is shown in Figure \ref{MILPmodel}.

\begin{figure}[!htbp]
\begin{tcolorbox}
\begin{equation}
\label{sS1-1}
\min\big(-cI_0+c\sum_{t=1}^T\tilde{d}_t+\sum_{t=1}^T(K \delta_t + hH_t + bB_t)+c\tilde{I}_T\big)
\end{equation}
Subject to, $t = 1, 2, \ldots, T$
\begin{align}
&\delta_t=0 \rightarrow \tilde{I}_{t}+ \tilde{d}_{t} - \tilde{I}_{t-1} = 0					\label{sS1-7}     \\
&\sum_{j=1}^t P_{jt}= 1													\label{sS1-9}	\\
&P_{jt} \geq \delta_{j} - \sum_{k=j+1}^t \delta_{k}, 				&j = 1, 2, \ldots, t	\label{sS1-10}	\\
&P_{jt} = 1 \rightarrow H_{t}= \hat{L}(\tilde{I}_t+\tilde{d}_{jt},d_{jt}),	&j = 1, 2, \ldots, t	\label{sS1-3}	\\
&P_{jt} = 1 \rightarrow B_{t}= L(\tilde{I}_t+\tilde{d}_{jt},d_{jt}),   		&j = 1, 2, \ldots, t	\label{sS1-4}	\\
&P_{jt} \in \{0, 1\},       									&j = 1, 2, \ldots, t   	\label{sS1-11}  	\\
&\delta_{t} \in \{0, 1\}														\label{sS1-12}	
\end{align}
\end{tcolorbox}
\caption{The formulation of the non-stationary stochastic lot-sizing problem}
\label{MILPmodel}
\end{figure}

The objective function (\ref{sS1-1}) computes the minimised expected total cost comprising ordering cost, holding cost and penalty cost. Constraints (\ref{sS1-7}) state inventory balance equations. Constraints (\ref{sS1-9}) indicate the most recent replenishment before period $t$ was issued in period $j$. Constraints (\ref{sS1-10}) identify uniquely the period in which the most recent replenishment prior to $t$ took place. Constraints (\ref{sS1-3}) and (\ref{sS1-4}) model end of period $t$ expected excess inventory and back-orders by means of the first order loss function.

We now discuss how to adapt the model in Fig. \ref{MILPmodel} in order to approximate $G_t(y)$. We call this modified model MINLP-$s$, and use superscript ``$s$" to label decision variables in this model. For any given initial inventory level $I_0^s$, let $G^s_1(I_0^s)$ denote the expected total cost over periods $1, \ldots, T$ without issuing an order in period $1$,
\begin{equation}\label{obj5-2}
G^s_1(I^s_0) = -cI_0^s+c\sum_{t=1}^T\tilde{d}_t+\sum_{t=1}^T(K \delta_t^s + hH_t^s + bB_t^s)+c\tilde{I}_T^s.
\end{equation}
MINLP-$s$ optimises $G^s_1(I^s_0)$ subject to constraints in Fig. \ref{MILPmodel} with an additional constraint
\begin{equation}\label{re4}
\delta_1^s=0,
\end{equation} 
which forces the model not to place a replenishment in period 1. Note that MINLP-$s$ can easily be approximated as an MILP model by using the approach discussed in \cite{rkt2015} to piecewise linearise loss functions in constraints (\ref{sS1-3}) and (\ref{sS1-4}). For further details please refer to \ref{oplsyntax}.

{\bf Example. }
In Fig. \ref{Gsfunction}, we plot the expected total cost $G_1^s(y)$ for the same $4$-period numerical example in Fig. \ref{Gfunction} with initial inventory level $I_0^s \in (0, 200)$, $G_1^s(y)$ are obtained via the MILP-$s$. Since $G_1^s(y)$ approximates $G_1(y)$, we can use $G_1^s(y)$ to find approximate values $\hat{S}_1$ and $\hat{s}_1$  for $S_1$ and $s_1$.

\begin{figure}[!htbp]
\centering
\begin{tikzpicture}[x=0.0380952380952381cm, y=0.018518518518518517cm]
\draw [-latex] ([xshift=-0mm] 0.0,0) -- ([xshift=5mm] 210.0,0) node[above] {Opening inventory level};
\draw (0.0,0) -- +(0mm,1.5mm) -- +(0mm,-1.5mm) node[below] {0};
\draw (25.0,0) -- +(0mm,1.5mm) -- +(0mm,-1.5mm) node[below] {25};
\draw (50.0,0) -- +(0mm,1.5mm) -- +(0mm,-1.5mm) node[below] {50};
\draw (75.0,0) -- +(0mm,1.5mm) -- +(0mm,-1.5mm) node[below] {75};
\draw (100.0,0) -- +(0mm,1.5mm) -- +(0mm,-1.5mm) node[below] {100};
\draw (125.0,0) -- +(0mm,1.5mm) -- +(0mm,-1.5mm) node[below] {125};
\draw (150.0,0) -- +(0mm,1.5mm) -- +(0mm,-1.5mm) node[below] {150};
\draw (175.0,0) -- +(0mm,1.5mm) -- +(0mm,-1.5mm) node[below] {175};
\draw (200.0,0) -- +(0mm,1.5mm) -- +(0mm,-1.5mm) node[below] {200};
\draw [-latex] ([yshift=-0mm] 0,0.0) -- ([yshift=3mm] 0, 270.0) node[right] {Expected total cost};
\draw (0,0.0) -- +(1.5mm,0mm) -- +(-1.5mm,0mm) node[left] {250};
\draw (0,100.0) -- +(1.5mm,0mm) -- +(-1.5mm,0mm) node[left] {350};
\draw (0,200.0) -- +(1.5mm,0mm) -- +(-1.5mm,0mm) node[left] {450};

\draw [color=black, mark= , style=dashed] (70,16.298) -- (70,-50) node[left] {$\hat{S}_1=70$};
\draw [color=black, mark= , style=dashed] (15,116.298) -- (15,-50) node[left] {$\hat{s}_1=15$};

\draw [color=black, mark= , style=solid, |-|] (70,116.298) node[right] {$G_1^s(\hat{s}_1)=366.298$} -- (70,16.298) node[right] {$G_1^s(\hat{S}_1)=266.298$} node[pos=0.5,right]{$K=100$};

\draw [color=black, mark= , style=dashed] (70,16.298)  -- (-1.5mm,16.298);
\draw [color=black, mark= , style=dashed] (70,116.298) -- (-1.5mm,116.298);

\draw [smooth, color=red, mark= , style=solid] plot coordinates {
(0,	261.692)
(1,	251.692)
(2,	241.692)
(3,	231.692)
(4,	221.692)
(5,	211.692)
(6,	201.692)
(7,	191.692)
(8,	181.692)
(9,	171.692)
(10, 162.002)
(11, 152.465)
(12, 142.928)
(13, 133.39)
(14, 124.762)
(15, 116.145)
(16, 108.248)
(17,	 100.849)
(18, 94.338)
(19, 88.344)
(20, 83.634)
(21, 79.344)
(22, 76.338)
(23, 73.849)
(24, 72.248)
(25, 71.145)
(26, 70.762)
(27,	 70.39)
(28, 70.928)
(29, 71.465)
(30, 72.002)
(31, 72.692)
(32, 73.692)
(33, 74.692)
(34, 75.692)
(35, 76.692)
(36, 77.692)
(37, 78.692)
(38, 79.692)
(39, 80.692)
(40, 81.692)
(41, 82.692)
(42, 83.692)
(43, 84.692)
(44, 85.692)
(45, 86.692)
(46, 87.692)
(47, 88.692)
(48, 89.692)
(49, 90.692)
(50, 86.298)
(51, 79.898)
(52, 73.499)
(53, 67.1)
(54, 60.701)
(55, 55.548)
(56, 50.553)
(57, 45.559)
(58, 40.564)
(59, 36.949)
(60, 33.449)
(61, 29.949)
(62, 26.564)
(63, 24.559)
(64, 22.553)
(65, 20.548)
(66, 18.701)
(67,	 18.1)
(68, 17.499)
(69, 16.898)
(70, 16.298)
(71, 16.591)
(72, 17.208)
(73, 17.826)
(74, 18.443)
(75, 19.06)
(76, 20.021)
(77,	 21.559)
(78, 23.096)
(79, 24.633)
(80, 26.171)
(81, 27.708)
(82, 29.245)
(83, 30.783)
(84, 32.385)
(85, 34.385)
(86, 36.385)
(87,	 38.385)
(88, 40.385)
(89, 42.385)
(90, 44.385)
(91, 46.385)
(92, 48.385)
(93, 50.385)
(94, 52.385)
(95, 54.385)
(96, 56.385)
(97, 58.385)
(98, 60.385)
(99, 62.385)
(100, 64.385)
(101, 66.385)
(102, 68.385)
(103, 70.385)
(104, 72.385)
(105, 74.385)
(106, 76.385)
(107, 78.385)
(108, 80.385)
(109, 82.385)
(110, 84.385)
(111, 86.385)
(112, 88.385)
(113, 90.385)
(114, 92.385)
(115, 94.385)
(116, 96.385)
(117, 98.385)
(118, 100.385)
(119, 102.385)
(120, 104.385)
(121, 106.385)
(122, 107.389)
(123, 104.889)
(124, 103.558)
(125, 102.553)
(126, 101.547)
(127, 100.542)
(128, 99.537)
(129, 98.531)
(130, 97.735)
(131, 98.134)
(132, 98.533)
(133, 98.932)
(134, 99.331)
(135, 99.731)
(136, 100.13)
(137, 100.529)
(138, 101.929)
(139, 103.547)
(140, 105.164)
(141, 106.781)
(142, 108.399)
(143, 110.016)
(144, 111.633)
(145, 113.251)
(146, 114.868)
(147, 117.269)
(148, 119.806)
(149, 122.344)
(150, 124.881)
(151, 127.418)
(152, 129.956)
(153, 132.493)
(154, 135.03)
(155, 137.568)
(156, 135.064)
(157, 132.576)
(158, 130.613)
(159, 128.651)
(160, 126.724)
(161, 125.224)
(162, 123.724)
(163, 122.224)
(164, 121.249)
(165, 121.244)
(166, 121.238)
(167, 121.233)
(168, 121.228)
(169, 121.222)
(170, 121.217)
(171, 121.212)
(172, 122.372)
(173, 123.771)
(174, 125.17)
(175, 126.569)
(176, 127.968)
(177, 129.367)
(178, 130.767)
(179, 132.166)
(180, 134.201)
(181, 136.818)
(182, 139.435)
(183, 142.053)
(184, 144.67)
(185, 147.287)
(186, 149.905)
(187, 152.522)
(188, 155.139)
(189, 157.757)
(190, 160.697)
(191, 164.234)
(192, 167.771)
(193, 171.309)
(194, 174.846)
(195, 178.383)
(196, 181.921)
(197, 185.458)
(198, 188.995)
(199, 192.533)
(200, 196.07)}
node[above right] {$G_1^s(y)$};
\end{tikzpicture}
\caption{Plot of $G_1^s(y)$}
\label{Gsfunction}
\end{figure}
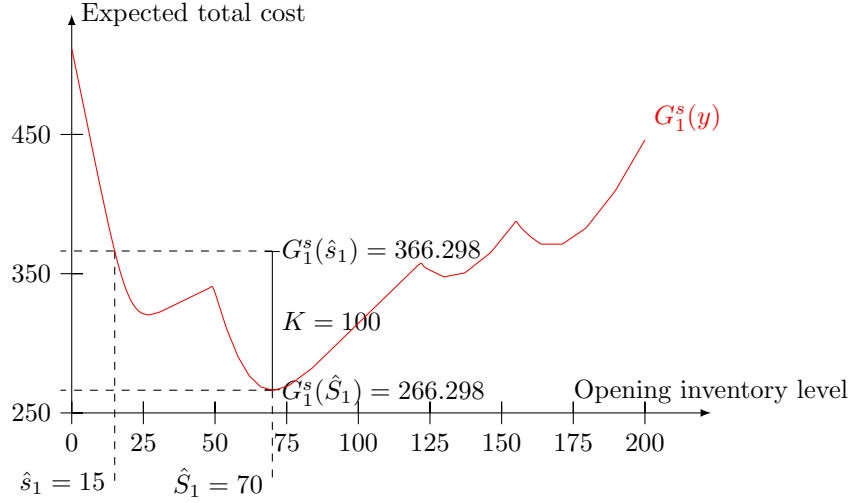

\section{An MINLP-based model to approximate $(s, S)$ policy parameters}\label{milpmodel}
\noindent
In this section we present an MINLP heuristic for computing near-optimal $(s,S)$ policy parameters. To the best of our knowledge, this is the first MINLP model for computing near-optimal $(s, S)$ policy parameters.

In a similar fashion to ``MINLP-$s$'', we introduce ``MINLP-$S$''. MINLP-$S$ imposes the constraint 
\begin{equation}\label{delta11}
\delta_1^S=1, 
\end{equation}
which forces the model to place a replenishment in period 1. Similarly to Eq. \ref{obj5-2}, let the objective function of MINLP-$S$ be $C_1^S(\cdot)$,
which approximates $C_1(\cdot)$. 

Recall that $I_0^S$ represents the initial inventory level in MINLP-$S$. Since in MINLP-$S$ a replenishment is forced in period 1 (Eq. \ref{delta11}), this variable --- which is left free to vary in the model --- represents an approximation $\hat{S}_1$ of the order-up-to-level $S_1$ in period $1$. We observe that $C_1^S(\hat{S}_1) = G_1^s(\hat{S}_1)+K$, since the only difference between MINLP-$S$ and MINLP-$s$ is the constraint that prescribes whether to force or not a replenishment in period $1$.

Since $G^s_1(y)$ is an approximation of $G_1(y)$,  if we identify an opening inventory level $I_0^s<\hat{S}_1$ such that $G_1^s(I_0^s)=G_1^s(I_0^S)+K$, then $s_1\approx I_0^s$. Therefore, we can approximate $s_1$ and $S_1$ simultaneously by connecting MINLP-$S$ and MINLP-$s$ via the constraint
\begin{equation}\label{re3}
G^s_1(I_0^s)=C_1^S(I_0^S).
\end{equation}
Finally, since $s_1\leq S_1$, we introduce an additional constraint to ensure that the reorder point is not greater than the order-up-to-level,
\begin{equation}\label{re7}
I_0^s \leq I_0^S.
\end{equation}

Note that, in contrast to the true value $G_1(y)$, there is no guarantee that $K$-convexity holds for its approximation $G^s_1(y)$. For some instances we may therefore have multiple values $s_1$ such that (\ref{re3}) holds. As we will demonstrate in our computational study, leaving to the solver the freedom to choose one of such values in a non-deterministic fashion leads to competitive results. 

MINLP-$S$ and MINLP-$s$ are connected by Eq. (\ref{re3}), in such a way the order-up-to-level $S_1$, the reorder point $s_1$, and the optimal expected total cost are approximated simultaneously. For the joint MINLP model, decision variables are those in both MINLP-$S$ and MINLP-$s$ with addition of initial inventory levels $I_0^S$ and $I_0^s$. The holistic objective function is to minimise the expected total cost of MINLP-$S$ over the planning horizon and the expected total cost of MINLP-$s$ from period two to the end of the planing horizon,
\begin{equation}
\begin{aligned}
\min\Big(&-cI_{0}^S + c\sum_{t=1}^T\tilde{d}_{t}^S+\sum_{t=1}^T(K \delta_t^S + hH_t^S + bB_t^S) +c\tilde{I}_T^S\\&-cI_{0}^s + c\sum_{t=1}^T\tilde{d}_{t}^s+\sum_{t=2}^T(K \delta_t^s + hH_t^s + bB_t^s) +c\tilde{I}_T^s \Big);
\end{aligned}
\end{equation}
note that the missing period for MINLP-$s$ is taken care of by constraints \ref{re4} and \ref{re3}.

Constraints of the joint MINLP model are those of both MINLP-$S$ and MINLP-$s$ in addition to the linking constraints (\ref{re4}), (\ref{delta11}), (\ref{re3}) and (\ref{re7}). By solving the joint MINLP model over the planning horizon $k, \ldots, T$, one estimates $S_k$ and $s_k$, where $k=1, \ldots, T$. As previously discussed, the joint MINLP model can also be linearised via the piecewise-linear approximation proposed in \cite{rkt2015}. In our MILP model, (\ref{sS1-3}) and (\ref{sS1-4}) are modelled via the piecewise OPL expression \citep{oplmanual}. For a complete overview of the MILP model refer to \ref{piecewisemodel}. 


{\bf Example. }
We now use the same $4$-period numerical example in Fig. \ref{Gsfunction} to demonstrate the modelling strategy behind the joint MINLP heuristic. We observe that, for period $1$, the approximated order-up-to-level is $S_1=70.2658$, the reorder point is $s_1=15.0008$, the optimal expected total cost $G_1^s(s_1)= 366.138$ as shown in Fig. \ref{Gfunction}. By solving the joint MNILP repeatedly, $S_t$, $s_t$ and $G^s_t(s_t)$, for $t=1, \ldots, 4$, are estimated as shown in Table \ref{minlpdetailparameter}.
\begin{table}[!htbp]
\centering
\begin{tabular}{|l|cccc|}
\hline
t &1&2&3&4\\
\hline
$s_t$ & 15.0008 & 29.0161 & 58.1089 & 29.0161\\
$S_t$ & 70.2658 & 53.9768 & 116.5530 & 53.9768\\
$G^s_t(s_t)$ & 366.138 &311.369 &193.338 & 118.031\\
\hline
\end{tabular}
\caption{Near-optimal $(s, S)$ policy parameters obtained via the joint MINLP heuristic}
\label{minlpdetailparameter}
\end{table}

\section{A binary search approach to approximate $(s, S)$ policy parameters}\label{binarysearchmodel}
\noindent
The joint MINLP heuristic presented in the last section can only effectively tackle small-size instances. In order to tackle larger-size problems, we introduce a more efficient approach that combines the model MINLP-$s$ discussed in Section \ref{milpmodel} and a binary search strategy. More precisely, we first let $I_0^s$ to be a decision variable in MINLP-$s$ and minimise $G_k^s(I_0^s)$ to estimate the order-up-to-level $\hat{S}_1$ and the minimised expected total cost $G_1^s(\hat{S}_1)$ for period $1$. Next, since the $K$-convexity holds for $G_1(y)$, there exits a unique reorder point $s_1$ such that $G_1(s_1)=G_1(S_1)+K$. Since $G_1^s(I^s_0)$ is an approximation of $G_1(y)$, we can conduct a binary search to approximate the reorder point $\hat{s}_1$ by $I_0^s \leq \hat{S}_1$ at which $G_1^s(I^s_0)=G_1^s(\hat{S}_1)+K$. By repeating this procedure over the planning horizon $k, \ldots, T$, we find pairs of $S_k$ and $s_k$, where $k =1, \ldots, T$. 

Algorithm \ref{algorithm} shows the binary search approach. For any given planning horizon $k, \ldots, T$, where $k=1,\ldots, T$, we first let $I_{k-1}^s$ to be a decision variable in MINLP-$s$ and minimise $G_k^s(I_{k-1}^s)$ so that to estimate the order-up-to-level $\hat{S}_k$ and the minimised expected total cost $G_k^s(\hat{S}_k)$ for period $k$. We assume, for the binary search method, the initial low value ($low$) is a large negative integer and the initial high value ($high$) is equal to $\hat{S}_k$ (line $4$ in Algorithm \ref{algorithm}). Then, we start the binary search procedure (line $5$) while $low < high$. We calculate the average value $mid=low+\text{round}((high-low)/2)$ (Line $6$). Next step is to run the MINLP-$s$ by updating the initial inventory level $I_{k-1}^s$ with the calculated middle value $I_{k-1}^s=mid$ and to obtain current expected total cost $G_k^s(I_{k-1}^s)$ (line $8$). If current cost $G_k^s(I_{k-1}^s)-G^s_k(\hat{S}_k)-K<0$, then we update $high=low-stepsize$ (line $10$); if current cost $G_k^s(I_{k-1}^s)-G_k^s(\hat{S}_k)-K>0$, then we update $low=mid+stepsize$ (line $12$); otherwise, $\hat{s}_k=mid$ (line $14$). By repeating this procedure over planning horizon $k, \ldots, T$, we obtain $\hat{s}_k$, $\hat{S}_k$, and the optimal cost, where $k=1, \ldots, T$. 
\begin{algorithm}[!htb]
\SetAlgoLined
\KwData{costs (ordering cost, holding cost, penalty cost), mean demand and standard deviation of each period, stepsize}
\KwResult{pairs of $s$ and $S$ for each period}
\BlankLine
\For{$k=1$ \KwTo $T$}{
Minimising MINLP-$s$ in Section \ref{milpmodel} in OPL\;
Obtaining $G_k^s(\hat{S}_k)$ and $\hat{S}_k$\;
\BlankLine
$low = \text{a large negative integer}$; $high=\hat{S}_k$\;
\BlankLine
\While{$low<high$}{
$mid=low+\text{round}((high-low)/2)$\;
Running the MINLP-$s$ with $I_{k-1}^s=  $ in OPL\;
Obtaining currentcost $G_k^s(I_{k-1}^s)$\;
\BlankLine
\If{$G_k^s(I_{k-1}^s)-G_k^s(\hat{S}_k)-K<0.0001$}{
$high=mid-\text{stepsize}$\;
\ElseIf{$G_k^s(I_{k-1}^s)-G_k^s(\hat{S}_k)-K>0.0001$}{
$low=mid+\text{stepsize}$\;
\Else{$\hat{S}_k=mid$\;
$low=high$\;
}
}
}
}
}
\caption{The binary search algorithm}
\label{algorithm}
\end{algorithm}

{\bf Example. }
We illustrate the solution method just discussed via the same $4$-period numerical example presented in Fig. \ref{Gfunction}. We assume the step size of the binary search is $0.01$. We observe that the order-up-to-level $\hat{S}_1=70.2658$ and the expected total cost $G_1^s(70.2658)=266.298$. We then set $low=-200$, $high=70.2658$. While $low < high$, the $mid$ is updated via the comparison of $G_1^s(I_0^s)$ and $G_1^s(70.2658)$. After a number of iterations, we obtain the reorder point $\hat{s}_1=15$ at which $G_1^s(15)=G_1^s(70.2658)$.  By repeating this procedure we obtain $\hat{S}_t$, $\hat{s}_t$, and $G^s_t(s_t)$, for each period $t=1, \ldots,4$ as displayed in Table \ref{binarysearhdetail}.

\begin{table}[!htbp]
\centering
\begin{tabular}{|l|cccc|}
\hline
t &1&2&3&4\\
\hline
$s_t$ & 15 & 29.01 & 58.1 & 29.01\\
$S_t$ & 70.2658 & 53.9768 & 116.5530 & 53.9768\\
$G^s_t(s_t)$ & 366.138 &311.369 &193.338 & 118.031\\
\hline
\end{tabular}
\caption{Near-optimal $(s, S)$ policy parameters obtained via the binary search approach}
\label{binarysearhdetail}
\end{table}

\section{Computational experience}\label{computationalstudy}
\noindent
In this section we present an extensive analysis of the heuristics discussed in Sections \ref{milpmodel} (MP) and \ref{binarysearchmodel} (BS). We first design a test bed featuring instances defined over an $8$-period planning horizon. On this test bed, we assess the behaviour of the optimality gap and the computational efficiency of both the MP and BS heuristics. Then we assess the computational performance of our the BS heuristics on a test bed featuring larger instances on a $25$-period planning horizon. For all cases, MINLP models are solved by employing the piecewise linearization strategy discussed in \cite{rkt2015}, which can be easily implemented in OPL by means of the \texttt{piecewise} syntax. Numerical examples are conducted by using the IBM ILOG CPLEX Optimization Studio 12.7 and MATLAB R2014a on a 3.2GHz Intel(R) Core(TM) with 8GB of RAM. 

\subsection{An $8$-period test bed}\label{eightperiodtestbed}
\noindent
We consider a test bed which includes $270$ instances. Specifically, we incorporate ten demand patterns displayed in Fig. \ref{demandpatterns}. These patterns comprising two life cycle patterns (LCY1 and LCY2), two sinusoidal patterns (SIN1 and SIN2), a stationary pattern (STA), a random pattern (RAND), and four empirical patterns (EMP1, ..., EMP4). Full details on the experimental setup are given in \ref{testbed}. Fixed ordering cost $K$ ranges in $\{200, 300, 400\}$, the penalty cost $b$ takes values $\{5, 10, 20\}$. We assume that demand $d_t$ in each period $t$ is independent and normally distributed with mean $\tilde{d_t}$ and coefficient of variation $c_v \in \{0.1, 0.2, 0.3\}$; note that $\sigma_t=c_v\tilde{d_t}$. Since we operate under the assumption of normality, our models can be readily linearised by using the piecewise linearisation parameters available in \cite{rtph2014}. However, the reader should note that our proposed modeling strategy is distribution independent, see \cite{rkt2015}.
\begin{figure}[!htbp]
\centering
\subfigure{
 \includegraphics[scale=0.55]{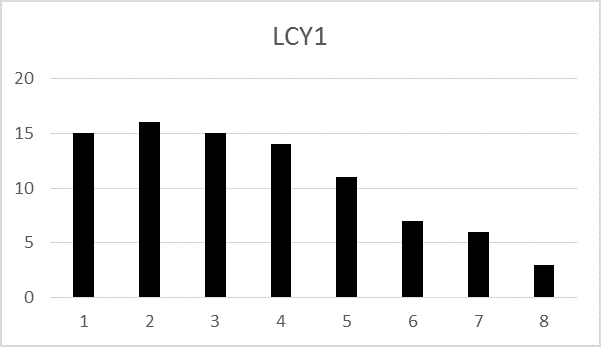}}
 \hspace{0.035in}   
\subfigure{
 \includegraphics[scale=0.55]{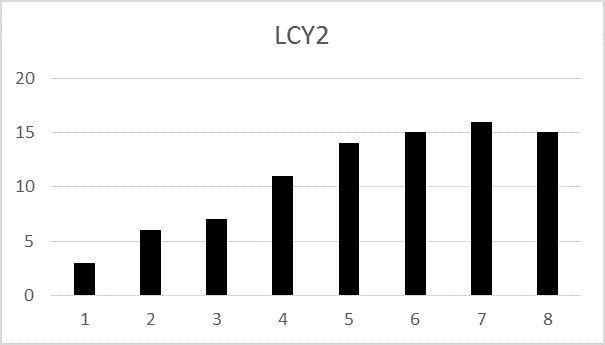}}
\subfigure{
 \includegraphics[scale=0.55]{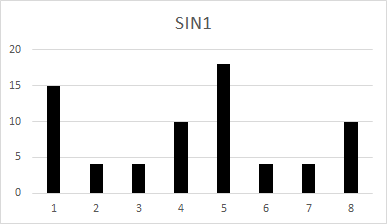}}
 \hspace{0.035in}
\subfigure{
 \includegraphics[scale=0.55]{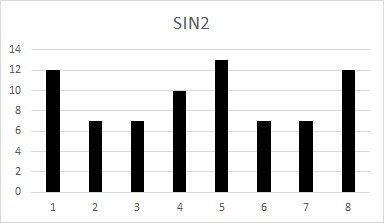}}
\subfigure{
 \includegraphics[scale=0.55]{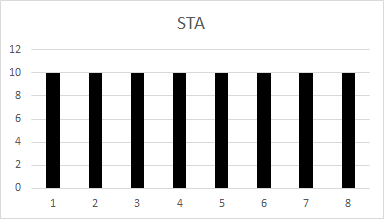}}
 \hspace{0.035in}
\subfigure{
 \includegraphics[scale=0.55]{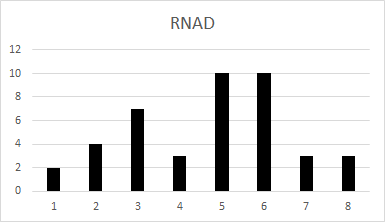}}
\subfigure{
 \includegraphics[scale=0.55]{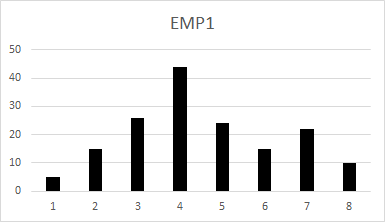}}
 \hspace{0.035in}
\subfigure{
 \includegraphics[scale=0.55]{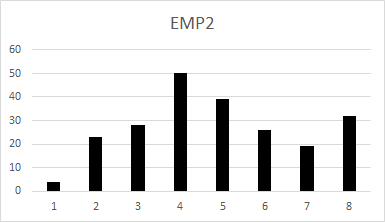}}
\subfigure{
 \includegraphics[scale=0.55]{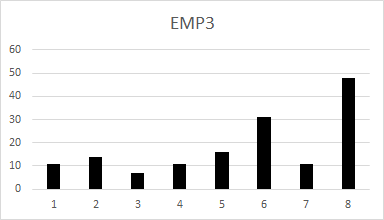}}
 \hspace{0.035in}
\subfigure{
 \includegraphics[scale=0.55]{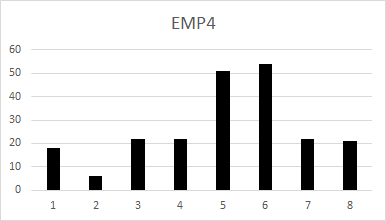}}
\caption{Demand patterns in our computational analysis}
\label{demandpatterns}
\end{figure}

We set the SDP model discussed in Section \ref{problemdescription} as a benchmark. We compare against this benchmark in terms of optimality gap and computational time.  First of all, we obtain optimal parameters for each test instance by implementing an SDP algorithm in MATLAB. Then, we solve each instance by adopting both modelling heuristics presented in Section \ref{milpmodel} and \ref{binarysearchmodel}. Specifically, for the MP heuristic we employ six segments in the piecewise-linear approximations of $B_{t}$ and $H_{t}$ (for $t=1,\ldots, T$) in order to guarantee reasonable computational performances; for the BS heuristic, whose computational performance is only marginally affected by an increased number of segments in the linearisation, we employ eleven segments and a step size $0.1$. To estimate the cost of the policies obtained via our heuristics, we simulate all policies via Monte Carlo Simulation (10,000 replications). 

Table \ref{shortperiodsexample} gives an overview of optimality gaps in terms of modelling methods and parameter settings. Both heuristics perform better when demand pattern is rather steady. It is difficult to make a general remark with respect to fixed ordering cost. Both methods perform worse as penalty cost increases. More specifically, when penalty cost increases from $10$ to $20$, the optimal gap rises from $0.28\%$ to $0.38\%$ and from $0.25\%$ to $0.44\%$, respectively. Similarly, performance of these two methods deteriorates as demand variability increases: optimality gap of the BS heuristic increases significantly from $0.18\%$ to $0.39\%$ as the coefficient of variation increases from $0.1$ to $0.3$. Overall, the average optimality gap of the MP heuristic is $0.33\%$, and that of the BS method is $0.28\%$. This discrepancy ought to be expected, since in the case of the BS method a higher number of segments has been employed.

\begin{table}[!htbp]
\centering
\begin{tabular}{|lcc|}
\hline
Modelling methods & MP & BS\\
\hline
\multicolumn{3}{|l|}{Demand pattern}\\
LCY1&0.28&0.39\\
LCY2&0.26&0.15\\
SIN1&0.18&0.14\\
SIN2&0.17&0.16\\
STA&0.25&0.23\\
RAND&0.14&0.16\\
EMP1&0.41&0.36\\
EMP2&1.01&0.78\\
EMP3&0.17&0.17\\
EMP4&0.44&0.21\\
\hline
\multicolumn{3}{|l|}{Fixed ordering cost}\\
200 &0.32&0.28\\
300&0.29&0.20\\
400&0.38&0.34\\
\hline
\multicolumn{3}{|l|}{Penalty cost}\\
5&0.19&0.14\\
10&0.28&0.25\\
20&0.38&0.44\\
\hline
\multicolumn{3}{|l|}{Coefficient of variation}\\
0.1&0.22&0.18\\
0.2&0.32&0.25\\
0.3&0.46&0.39\\
\hline
Average gap&0.33&0.28\\
\hline
\end{tabular}
\label{shortperiodsexample}
\caption{Average optimality gaps $\%$ of the 8-period test bed with different setting parameters and modelling methods}
\end{table}

Existing heuristics \cite{a1981} and \cite{bm1999} were reimplemented by \cite{dktr2016} and assessed on a test bed that neatly resembles the one adopted in this work. As shown in \cite{dktr2016}, Askin's optimality gap is $3.9\%$, and Bollapragada and Morton's is $4.9\%$. The optimality gap of our heuristic is $0.33\%$ when six segments are employed in the piecewise linearisation, and it drops to $0.28\%$ when eleven segments are employed. Our models therefore outperform both \cite{a1981} and \cite{bm1999} in terms of optimality gap on the test bed here considered.

Table \ref{shortperiodtime} shows computational times with regard to different setting parameters and modelling methods. Note "STDEV" in Table \ref{shortperiodtime} represents the standard deviation. The average computational time of the MP heuristic is $51.01s$, that of the BS method is $7.64s$, and that of the SDP model is $60.21s$. The computational times of the SDP and of the MP model vary significantly for different demand patterns considered, while that of the BS method remains stable. In particular, when the demand setting is EMP3, the average computational time of the MP model is $286.21s$; whereas, when the demand setting is EMP4, it is just $22.41s$. We observe that fixed ordering cost, penalty cost, and coefficient of variation do not have considerable influence on computational time of small-scale instances. Additionally, standard deviation of the MP model and of the SDP model fluctuate significantly, while that of the BS tend to remain stable. 

\begin{table}[!htbp]
\centering
\begin{tabular}{|lcccccc|}
\hline
\multirow{2}{*}{Settings} &\multicolumn{2}{|c}{MP} & \multicolumn{2}{|c}{BS}&\multicolumn{2}{|c|}{SDP}\\
\cline{2-7}
&\multicolumn{1}{|c}{Mean}&\multicolumn{1}{|c}{STDEV}&\multicolumn{1}{|c}{Mean}&\multicolumn{1}{|c}{STDEV}&\multicolumn{1}{|c}{Mean}&\multicolumn{1}{|c|}{STDEV}\\
\hline
\multicolumn{7}{|l|}{Demand pattern}\\
LCY1&4.07&0.81&8.22&0.66&14.42&0.03\\
LCY2&25.73&47.36&8.15&0.76&14.41&0.03\\
SIN1&3.88&0.74&6.90&0.64&14.41&0.02\\
SIN2&3.85&0.62&6.70&0.70&14.37&0.08\\
STA&9.18&21.13&6.84&0.63&7.69&0.05\\
RAND&3.48&0.51&7.48&1.07&7.50&0.06\\
EMP1&53.32&140.72&8.00&0.82&150.13&1.12\\
EMP2&97.99&162.94&8.17&0.77&114.44&1.31\\
EMP3&286.21&636.73&7.49&0.72&114.46&1.09\\
EMP4&22.41&40.25&8.45&0.89&150.24&0.35\\
\hline
\multicolumn{7}{|l|}{Fixed ordering cost}\\
200&88.81&365.45&7.71&0.97&60.17&59.96\\
300&33.75&99.76&7.62&0.92&60.29&60.07\\
400&30.48&109.84&7.59&1.07&60.16&59.99\\
\hline
\multicolumn{7}{|l|}{Penalty cost}\\
5&81.62&343.68&7.44&0.93&60.34&60.14\\
10&51.78&182.68&7.56&0.86&60.24&60.03\\
20&19.63&65.62&7.92&1.06&60.04&59.83\\
\hline
\multicolumn{7}{|l|}{Coefficient of variation}\\
0.1 &39.22&165.33&7.66&1.00&60.23&60.01\\
0.2&76.09&348.42&7.66&0.91&60.18&59.98\\
0.3&37.73&89.68&7.59&1.05&60.20&60.03\\
\hline
Average&51.01&51.01&7.64&0.99&60.21&60.21\\
\hline
\end{tabular}
\caption{Average computational times (seconds) of the 8-period test bed with different setting parameters and modelling methods}
\label{shortperiodtime}
\end{table}

\subsection{A $25$-period test bed}
\noindent
As shown in Section \ref{eightperiodtestbed} for the $8$-period test bed, both the MP and the BS methods provide tight optimality gaps and acceptable computational efficiency. We now extend the $8$-period test bed to $25$ periods with larger instances. Demands of LCY1, LCY2, SIN1, SIN2, STA,  and RAND are generated with expressions (\ref{LCY1}), (\ref{LCY2}), (\ref{SIN1}), (\ref{SIN2}), (\ref{STA}), and (\ref{RAND}) in Fig. \ref{expression}. Demands of EMP1, EMP2, EMP3 and EMP4 are derived from \cite{ssb2011}. Full details are given in \ref{testbed}. Assume that fixed ordering cost ranges in $\{500, 1000, 1500\}$, penalty cost takes values $\{5, 10, 20\}$, and the coefficients of standard deviations are $\{0.1, 0.2, 0.3\}$. 
\begin{figure}[!htbp]
\begin{tcolorbox}
\begin{align}
&d_t=\text{round}(\frac{190\times e^{-(t-13)^2}}{2\times5^2}),&t=1, 2, \ldots, T\label{LCY1}\\
&d_t=\text{round}(\frac{170\times e^{-(t-13)^2}}{2\times6^2}),&t=1, 2, \ldots, T\label{LCY2}\\
&d_t=\text{round}\Big(70\times \text{sin}(0.8t)+80\Big),&t=1, 2, \ldots, T\label{SIN1}\\
&d_t=\text{round}\Big(30\times \text{sin}(0.8t)+100\Big),&t=1, 2, \ldots, T\label{SIN2}\\
&d_t=100,& t = 1, 2, \ldots, T\label{STA}\\
&d_t=\text{round}(\text{random}(0, 250)),&t=  1, 2, \ldots, T\label{RAND}
\end{align}
\end{tcolorbox}
\caption{Expressions for generating demand data}
\label{expression}
\end{figure}

We obtain optimal $(s, S)$ parameters and record computational times obtained via the BS method. For the first $15$ periods we perform binary search with step size $1$ in order to ensure fast convergence; for the last $10$ periods, we adopt a step size $0.1$ to enhance accuracy. The number of segments used in the piecewise linearisation is eleven. To estimate the cost of the policy obtained via our approximation, we simulate each instance one million times in MATLAB. We summarise computational times in Table \ref{longperiodtime}. 
\begin{table}[!htbp]
\centering
\begin{tabular}{|lcc|}
\hline
Settings & Mean & standard deviation\\
\hline
\multicolumn{3}{|l|}{Demand pattern}\\
LCY1&588.18&213.91\\
LCY2&806.25&338.10\\
SIN1&579.45&181.66\\
SIN2&1767.06&688.88\\
STA&1933.07&760.81\\
RAND&458.99&120.79\\
EMP1&696.20&123.23\\
EMP2&201.08&36.72\\
EMP3&1054.01&316.17\\
EMP4&187.17&44.98\\
\hline
\multicolumn{3}{|l|}{Fixed ordering cost}\\
500 &1039.49&901.76\\
1000&844.54&583.64\\
1500&597.41&362.24\\
\hline
\multicolumn{3}{|l|}{Penalty cost}\\
5&792.97&615.24\\
10&871.05&749.53\\
20&817.42&663.10\\
\hline
\multicolumn{3}{|l|}{Coefficient of variation}\\
0.1&744.61&617.16\\
0.2&838.61&682.91\\
0.3&898.11&723.86\\
\hline
Average&827.15&679.02\\
\hline
\end{tabular}
\caption{BS heuristics on a $25$-period test bed, average computational times (seconds) with different setting parameters}
\label{longperiodtime}
\end{table}

According to Table \ref{longperiodtime}, the computational time drops dramatically from $1039.40s$ to $597.41s$ as the fixed ordering cost increases from $500$ to $1500$. In contrast, with the increase of  coefficient of variation, the computational times rise significantly. For instance, when the coefficient of variation rises from $0.1$ to $0.2$, the computational time increases from $744.61s$ to $838.61s$. Whereas, standard deviations are large for all test instances. On average, the computational time is $827.15s$ and the standard deviation is $679.02s$.

\section{Conclusion}\label{conclusion}
\noindent
In this paper we discussed two MINLP-based heuristics for tackling non-stationary stochastic lot-sizing problems under $(s, S)$ policy. These heuristics are based on mathematical programming models that can be solved by using off-the-shelf optimization packages. More specifically, we introduced the first MINLP model for computing near-optimal nonstationary ($s,S$) policy parameters and a binary search strategy to tackle larger-size problems. These MINLP models can be linearised via the approach discussed in \cite{rkt2015} and can be implemented in OPL by adopting the \texttt{piecewise} expression. 

We conducted an extensive computational study comprising $270$ instances. We considered ten demand patterns, three fixed ordering costs, three penalty costs and three coefficients of variation. 

For the 8-period numerical study, we investigated the performance of both models by contrasting costs of the policy obtained with our models against costs of the optimal policy obtained via the stochastic dynamic programming. Optimality gaps observed are generally below $0.3\%$. Our sensitivity analysis showed that the optimality gap is tighter when the demand keeps stable, and performance deteriorate with the increase of the penalty cost and the coefficient of variation; both models provide tighter gaps than those reported in the literature \citep{a1981,bm1999}. 

The computational study carried out on larger instances (25-period planning horizon) showed that the computational efficiency of the binary search approach is reasonable: around $827.15s$ on average. Our sensitivity analysis demonstrates that the computational time is positively correlated to the penalty cost and coefficient of demand variation, and has negative correlation with the fixed ordering cost.

\bibliography{elsarticle-template}
\appendix 
\section{The \texttt{piecewise} OPL constraint}\label{oplsyntax}
\noindent
\cite{rkt2015} piecewise linearised loss functions in constraints (\ref{sS1-3}) and (\ref{sS1-4}) by employing piecewise linear approximations based on Jesen's and Edmundson-Madanski inequalities. An alternative strategy is to model these non-linear functions by exploring the  \texttt{piecewise} syntax in OPL. By using this syntax, a piecewise function is specified by giving a set of slopes which represent  the linear variation for each linear segment; a set of breakpoints at which slopes change; and the function value at a known point. 

\begin{figure}[!htbp] 
\centering
\begin{tcolorbox}
\begin{verbatim}
piecewise(i in 1..W){
slope[i] -> breakpoint[i];
 slope[W+1]
}(<knownpoint>,<valuepoint>)<value>;
\end{verbatim}
\end{tcolorbox}
\caption{The syntax of the \texttt{piecewise} command in OPL}
\label{syntax}
\end{figure}

The \texttt{piecewise} syntax in OPL is given in Figure \ref{syntax}. \texttt{W} is the number of breakpoints of the piecewise function. \texttt{slope[i]} and \texttt{breakpoint[i]} denote slope and breakpoint of segment $i$. Segment $i$ goes from breakpoint ($i-1$) to breakpoint ($i$). \texttt{<valuepoint>} is the function value at a known point \texttt{<knownpoint>}. Finally, \texttt{<value>} represents the value at which we evaluate the function. 

For the OPL \texttt{piecewise} syntax, there are three key components: slope, breakpoint, and function value at a known point. The following lemmas will demonstrate how to deduce their values. Let $\Omega$ be the support of $\omega$. Let $(\Omega_i)_{i=1,\ldots, W+1}$ be a partition of $\Omega$ in $W+1$ segments.

\begin{mylemma}\label{slope}
The slope of $i^{th}$ segment is written as
\[l_i=\sum_{k=1}^{i-1}p_k, i  \in \{1,2, \ldots, W+1\},\]
where $p_i=Pr\{\omega \in \Omega_i\}=\int_{\Omega_i}g_{\omega}(t)dt$, $g_{\omega}(\cdot)$ denotes the probability density function of $\omega$.
\end{mylemma}
\begin{proof}
Observation from \cite{rtph2014}, Lemma 11.
\end{proof}

\begin{mylemma}\label{breakpoint}
The $i^{th}$ breakpoint can be written as 
\[
X_i=E[\omega|\Omega_i], i \in \{1,2, \ldots, W\}.
\]
\end{mylemma}
\begin{proof}
Observation from \cite{rtph2014}, Lemma 11.
\end{proof}

Note that when $\omega$ follows a normal distribution with mean $\mu$ and standard deviation $\sigma$, then $\hat{L}_{\text{up}}(x, \omega)=\sigma \hat{L}_{\text{up}}(\frac{x-\mu}{\sigma},Z)$,
where $Z$ follows a standard normal distribution, see Lemma 7 in \cite{rtph2014}.
\begin{mylemma}\label{functionvalue}
Assume that the partition of $\Omega$ is symmetric with respect to $0$, then the function value $\hat{L}_{\text{up}}(x, \omega)$ at point $0$ can be written as follows.
\[\hat{L}_{\text{up}}(0, \omega)=\begin{cases}
-\sum_{k=1}^{\frac{W+1}{2}}p_kE[\omega|\Omega_k]+e_W, &\text{W is odd}\\
-\frac{1}{2}(\sum_{k=1}^{\frac{W}{2}}p_kE[\omega|\Omega_k]+\sum_{k=1}^{\frac{W}{2}+1}p_kE[\omega|\Omega_k])+e_W, &\text{W is even}
\end{cases}
\]
where $e_W$ represents the approximation error. 
\end{mylemma}
\begin{proof}
Since the partition of $\Omega$ is symmetric when $W$ is odd, $x=0$ is the central breakpoint. Hence, the function value at this breakpoint can be calculated directly. However, when $W$ is even, the function value at point $x=0$ is the average of nearest two symmetric breakpoints $X_{\frac{W}{2}}$ and $X_{\frac{W}{2}+1}$. 
\end{proof}

Following Lemma \ref{slope}, \ref{breakpoint} and \ref{functionvalue}, constraint (\ref{sS1-3}) and (\ref{sS1-4}) in Fig. \ref{MILPmodel} can be expressed as Eq. (\ref{picewiseholding}) and (\ref{picewisepenalty}) in Fig. \ref{piecewisecostinopl}, for $t=1,\ldots,T$.
\begin{figure}[!htbp]
\begin{tcolorbox}

\begin{multline}
P_{jt}=1 \rightarrow H_{t}=\texttt{piecewise}\{l_i  \to X_i; 1\}(0,\hat{L}_{up}(0,d_{jt}))\tilde{I}_{t}, \\ i = 1,\ldots, W;\  j=1,\ldots,t.\label{picewiseholding}
\end{multline}
\begin{multline}
P_{jt}=1\rightarrow B_{t}=\texttt{piecewise}\{-1+l_i \to X_i; 0\}(0,\hat{L}_{up}(0, d_{jt}))\tilde{I}_{t}, \\  i = 1,\ldots, W;\  j=1,\ldots,t.\label{picewisepenalty}
\end{multline}

\end{tcolorbox}
\caption{Rewriting holding and penalty costs by adopting \texttt{piecewise} syntax}
\label{piecewisecostinopl}
\end{figure}

\section{The MILP model}\label{piecewisemodel}
\noindent
The joint MILP model to calculate near-optimal $(s, S)$ policy parameters for the non-stationary stochastic lot-sizing problem is presented below. Note that we plug in the original fomulations (\ref{jointmilp1}), (\ref{jointmilp2}), (\ref{jointmilp5}), and (\ref{jointmilp6}) to our joint MILP model in order to enhance the computational perforance without excessively compromising solution quality.

{\scriptsize

\begin{align}
\min\Big(&-cI_{0}^S + c\sum_{t=1}^T\tilde{d}_{t}^S+\sum_{t=1}^T(K \delta_t^S + hH_t^S + bB_t^S) +c\tilde{I}_T^S \notag \\
&-cI_{0}^s + c\sum_{t=1}^T\tilde{d}_{t}^s+\sum_{t=2}^T(K \delta_t^s + hH_t^s + bB_t^s) +c\tilde{I}_T^s \Big)
\end{align}
Subject to, $t = 1,\ldots, T$
\begin{align}
&C^S_t(I_{0}^S)=-cI_{0}^S+ c\sum_{t=1}^T\tilde{d}_{t}+\sum_{t=1}^T(K \delta_t^S + hH_t^S + bB_t^S) +c\tilde{I}_T^S\\
&\tilde{I}_{t}^S+ \tilde{d}_{t} - \tilde{I}_{t-1}^S \geq 0\\
&\tilde{I}_{t}^S+ \tilde{d}_{t} - \tilde{I}_{t-1}^S \leq \delta_{t}^SM\\
&\sum_{j=1}^tP_{jt}^S = 1\\
&P_{jt}^S \geq \delta_{j}^S - \sum_{k=j+1}^t \delta_{k},
 j = 1, \ldots, t
\end{align}
\begin{align}
&\delta_{1}^S=1\\
&I_0^S=\tilde{I}_1^S+\tilde{d}_1\\
&H_{t}^{S}\geq (I_t^S+\sum_{j=1}^td_{jt}P_{jt}^S)\sum_{k=1}^ip_k-\sum_{j=1}^t(\sum_{k=1}^ip_kE[d_{jt}|\Omega_i]-e_W)P_{jt}^S, & i = 1, \cdots, W \label{jointmilp1}\\
&B_{t}^{S}\geq  -I_t^S+(I_t^S+\sum_{j=1}^td_{jt}P_{jt}^S)\sum_{j=1}^ip_k-\sum_{j=1}^t(\sum_{k=1}^ip_kE[d_{jt}|\Omega_i]-e_W)P_{jt}^S,  &i = 1, \cdots, W \label{jointmilp2}\\
&P_{jt}^S \in \{0, 1\}, j = 1, \ldots, t\\
&\delta_{t}^S \in \{0, 1\}\\
&G_t^s(I_{0}^s)=-cI_{0}^s + (hH_1^s + bB_1^s) + c\sum_{t=1}^T\tilde{d}_{t}+\sum_{t=2}^T(K \delta_t^s + hH_t^s + bB_t^s) +c\tilde{I}_T^s\\
&\tilde{I}_{t}^s+ \tilde{d}_{t} - \tilde{I}_{t-1}^s \geq 0\\
&\tilde{I}_{t}^s+ \tilde{d}_{t} - \tilde{I}_{t-1}^s \leq \delta_{t}^sM\\
&\sum_{j=1}^tP_{jt}^s = 1\\
&P_{jt}^s \geq \delta_{j} - \sum_{k=j+1}^t \delta_{k}^s, j = 1, \ldots, t\\
&\delta_{1}^s=0\\
&P_{jt}^s=1\rightarrow H_{t}^{s}=\texttt{piecewise}\{l_i \to X_i; 1\}(0,\hat{L}_{up}(0, d_{jt}))\tilde{I}_{t}^s,	&\begin{array}{l} i = 1, \ldots, W\\  j=1,\ldots, t\end{array} \label{jointmilp3}\\
&P_{jt}^s=1\rightarrow B_{t}^{s}=\texttt{piecewise}\{-1+l_i  \to X_i; 0\}(0,\hat{L}_{up}(0, d_{jt}))\tilde{I}_{t}^s	&\begin{array}{l} i = 1, \ldots, W\\  j=1,\ldots, t\end{array}  \label{jointmilp4}\\
&H_{t}^{s}\geq (I_t^s+\sum_{j=1}^td_{jt}P_{jt}^s)\sum_{k=1}^ip_k-\sum_{j=1}^t(\sum_{k=1}^ip_kE[d_{jt}|\Omega_i]-e_W)P_{jt}^s,		&i = 1, \cdots, W \label{jointmilp5}\\
&B_{t}^{s}\geq  -I_t^s+(I_t^s+\sum_{j=1}^td_{jt}P_{jt}^s)\sum_{j=1}^ip_k-\sum_{j=1}^t(\sum_{k=1}^ip_kE[d_{jt}|\Omega_i]-e_W)P_{jt}^s,  	&i = 1, \cdots, W \label{jointmilp6}\\
&P_{jt}^s \in \{0, 1\},  j = 1,  \ldots, t \\
&\delta_{t}^s \in \{0, 1\}\\
&I_0^s \leq \tilde{I}_1^S+\tilde{d}_1\\
&G_1^s(I_0^s)=C_1^S(\tilde{I}_1^S+\tilde{d}_1)
\end{align}

}

\newpage

\section{Test bed}\label{testbed}
\noindent
Periodic demands with different demand patterns under the eight period computational study are displayed in Table \ref{demandpatterns1}. The demand of each period under the twenty-five periods numerical example is shown in Table \ref{demandpatterns2}. The first column represents period indexes; the rest columns denote various demands.\\
\begin{table}[!htbp]
\small
\centering
\begin{tabular}{|l|c c c c c c c c c c|}
\hline
Period&LCY1  &LCY2  &SIN1  &SIN2  &STA  &RAND   &EMP1  &EMP2  &EMP3  &EMP4 \\
\hline
1 & 15  & 3      &15     &12      &10    &2         &5         &4        &11      &18\\
2 & 16  & 6      &4       &7        &10    &4        &15        &23      &14      &6\\
3 & 15  & 7      &4       &7        &10    &7        &26        &28      &7        &22\\
4 & 14  & 11    &10      &10      &10    &3        &44        &50      &11      &22\\
5 & 11  & 14    &18      &13      &10    &10      &24        &39      &16      &51\\
6 & 7   & 15     &4       &7        &10    &10      &15        &26      &31      &54\\
7 & 6   & 16     &4       &7        &10    &3        &22        &19      &11      &22\\
8 & 3   & 15     &10     &12       &10    &3        &10        &32      &48      &21\\
\hline
\end{tabular}
\caption{Demand data of the 8-period computational analysis}
\label{demandpatterns1}
\end{table}

\begin{table}[!htbp]
\small
\centering
\begin{tabular}{|l|c c c c c c c c c c|}
\hline
Period&LCY1  &LCY2  &SIN1  &SIN2  &STA  &RAND   &EMP1  &EMP2  &EMP3  &EMP4 \\
\hline
1&11&23&130&122&100&178&2&47&44&49\\
2&17&32&150&130&100&178&51&81&116&188\\
3&26&42&127&120&100&136&152&236&264&64\\
4&38&55&76&98&100&211&467&394&144&279\\
5&53&70&27&77&100&119&268&164&146&453\\
6&71&86&10&70&100&165&489&287&198&224\\
7&92&103&36&81&100&47&446&508&74&223\\
8&115&120&88&103&100&100&248&391&183&517\\
9&138&136&136&124&100&62&281&754&204&291\\
10&159&150&149&130&100&31&363&694&114&547\\
11&175&161&121&118&100&43&155&261&165&646\\
12&186&168&68&95&100&199&293&195&318&224\\
13&190&170&22&75&100&172&220&320&119&215\\
14&186&168&11&71&100&96&93&111&482&440\\
15&175&161&42&84&100&69&107&191&534&116\\
16&159&150&96&107&100&8&234&160&136&185\\
17&138&136&140&126&100&29&124&55&260&211\\
18&115&120&148&129&100&135&184&84&299&26\\
19&92&103&114&115&100&97&223&58&76&55\\
20&71&86&60&91&100&70&101&0&218&0\\
21&53&70&18&73&100&248&123&0&323&0\\
22&38&55&14&72&100&57&99&0&102&0\\
23&26&42&50&87&100&11&31&0&174&0\\
24&17&32&104&110&100&94&82&0&284&0\\
25&11&23&144&127&100&13&0&0&0&0\\
\hline
\end{tabular}
\caption{Demand data of the 25-period computational analysis}
\label{demandpatterns2}
\end{table}

\end{document}